\documentclass[11pt]{article}
\usepackage{latexsym,color,amsmath,amsthm,amssymb,amscd,amsfonts,mathtools}

\usepackage[a4paper, left=1in, right=1in, top=1in, bottom=1in]{geometry}

\newtheorem{lemma}{Lemma}

\newtheorem{theorem}[lemma]{Theorem}
\newtheorem{definition}[lemma]{Definition}

\newtheorem{lem}[lemma]{Lemma}
\newtheorem{prop}[lemma]{Proposition}
\newtheorem{cor}[lemma]{Corollary}

\newtheorem{rem}[lemma]{Remark}

\newtheorem*{remark*}{Remark}

\def\S{{\cal S}}

\def\qed{\hfill $\vcenter{\hrule height .3mm
		\hbox {\vrule width .3mm height 2.1mm \kern 2mm \vrule width .3mm
			height 2.1mm} \hrule height .3mm}$ \bigskip}

\def\to{\rightarrow}

\newcommand{\iprod}[2]{\langle #1,#2 \rangle} 
\newcommand*\diff{\mathop{}\!\mathrm{d}}

\def\RR{\mathbb{R}}

\def\vol{{\rm Vol}}

\def\K{{\cal K}}
\def\FK{F_{c,\delta}}
\DeclareMathOperator*{\argmax}{arg\,max}

\DeclarePairedDelimiter{\parens}()

\newcounter{c}
\title{Some inequalities of isoperimetric type for the $c$-affine surface area}

\author{Shiri Artstein-Avidan, Arnon Chor \thanks{The authors are supported in part by the ERC under the European Union’s Horizon 2020 research and innovation programme (grant agreement no. 770127), by ISF grant Number 784/20, and by the Binational Science Foundation (grant no. 2020329).}}
\date{\today}

\begin{document}

\maketitle 
\begin{abstract}
    We study the $c$-affine surface area $\Omega^c(K)$, recently introduced by Sch\"{u}tt, Werner and Yalikun. We show that on the class of ball-bodies, $\Omega^c(K)$ is uniquely maximized by a ball of radius $\frac{n}{n+1}$, and that a Santal\'{o}-type inequality holds, and $\Omega^c(K)\Omega^c(K^c)$ is uniquely maximized be a Euclidean  ball of radius $1/2$. We also produce some more intricate inequalities involving the surface area.
\end{abstract}

In this short note we discuss a new notion reminiscent of the affine surface area of a convex body, which is relevant to the class of ball-bodies $\S_n$ (see \cite{AF-perrint, bezdek2024selected} and the references therein).
The class $\S_n$, which is a subset of the class $\K^n$ of all convex bodies in $\RR^n$ (i.e. compact, nonempty convex sets), consists of those convex bodies which are intersections of translates of the unit Euclidean ball in $\RR^n$. Equivalently, in addition to singletons, the class consists of convex bodies with generalized sectional curvature at least $1$ at every boundary point. In this note, in contrast with previous works, we exclude $\emptyset, \RR^n$ from $\S_n$. The unit Euclidean ball is denoted here by $B_2^n$. 

The class $\S_n$ (which is also sometimes called the $c$-class, because of the $c$-duality operation which we soon define) can also be characterized as the class of summands of $B_2^n$, namely bodies $K\in {\cal K}^n$ for which there exists $L\in {\cal K}^n$ with $K + L = B_2^n$. We have recently studied this class intensively \cite{AF-perrint}, emphasizing the role of the unique (up to trivial symmetries) order reversing bijection defined on this set, mapping a body $K\in \S_n$ to the unique $L\in \S_n$ with $K-L = B_2^n$. We also showed in \cite{ACF-preprint,artstein2025full} that it is the unique isometry (with respect to the Hausdorff metric) on this class, other than the identity map, again up to trivial symmetries. In geometric terms one may explicitly write this duality as 
\[ K^c = \bigcap_{x\in K} (x+B_2^n).\]
We note that for any $K \in \S_n$, $K^{cc} = K$.

Recently, Sch\"{u}tt, Werner and Yalikun in their paper ``Floating bodies for ball-convex bodies'' \cite{schuett2025floating} showed that in $\S_n$ there is a natural concept replacing the usual ``floating body'', where instead of considering half-spaces which cut off a portion at most $\delta$ of the volume, one intersects $1$-balls which do the same, thus forming a $c$-floating body which is in the class $\S_n$ and which converges to $K$ as $\delta\to 0$. We denote it here by $\FK  (K)\subset K$. Moreover, for bodies which have, at all points, sectional curvatures all strictly greater than $1$, they showed that the limit of the volume difference, scaled appropriately, exists, and computed a closed formula for it. They showed 
\begin{equation}\label{eq:asa-new} \lim_{\delta \to 0} \frac{\vol_n(K) - \vol_n(\FK (K))} {\delta^\frac{2}{n+1}}  = c_n \int_{\partial K} \prod _{i=1}^{n-1} \left( \kappa_i (K, x) -1 \right)^\frac{1}{n+1} \diff \mu_K(x)
\end{equation}
where $\diff \mu_K$ denotes the Lebesgue surface area measure on $\partial K$, $c_n = \frac12 \parens*{\frac{n+1}{\vol_{n-1}(B_2^{n-1})}}^{\frac{2}{n+1}}$ and $\kappa_i$ denote the sectional curvatures at the point $x\in \partial K$.

The expression \eqref{eq:asa-new} should be compared with the corresponding limit formula for the usual floating body, which gives, up to constant, the affine surface area of the body, given by 
\begin{equation} \Omega (K) = \int_{\partial K} \prod _{i=1}^{n-1}   \kappa_i (K, x)^\frac{1}{n+1} \diff \mu_K(x). 
\end{equation}

The conditions on the sectional curvatures of $K$ may seem restrictive. Nevertheless, for a convex body, almost every point on the boundary $\partial K$ is a ``normal'' point, namely the indicatrix and principal curvatures exist and are well defined (see \cite[Section 2.6]{schneider2013convex}). For a body in $\S_n$, these sectional curvatures are at least $1$. 
This fact, together with  \eqref{eq:asa-new}, motivated their definition for the functional, which we call $c$-affine-surface-area, although the class $\S_n$ itself is not affine invariant so the name is far from perfect; nevertheless it is a $c$-variant of the affine surface area. We mention that  in \cite{schuett2025floating} this functional is referred to as ``relative affine surface area'' and is denoted $as^R$.
\begin{definition}\label{def:asa-new}
Let $n \geq 2$ and $K\in \S_n$ which is not a point. Define the $c$-affine surface area of $K$ to be 
\begin{equation*}
\Omega^c(K) = \int_{\partial K} \prod _{i=1}^{n-1} \left( \kappa_i (K, x) -1 \right)^\frac{1}{n+1}  \diff \mu_K(x). 
\end{equation*}
\end{definition}

In what follows it will be convenient to introduce also the principal radii of curvature $r_i (u)$, defined as the non-zero eigenvalues (in increasing order) of $\nabla^2 h_K(u)$, the Hessian (in the sense of Alexandrov) of the support function $h_K(u) = \max_{y \in K} \iprod{y}{u}$. These are defined almost everywhere on $S^{n-1}$, see discussion in the Appendix and in \cite[Section 2.6 and Section 1.5, Note 3]{schneider2013convex}. It follows from the discussion in~\cite[Section 2.6, page 130]{schneider2013convex} or from~\cite{hug2002absolute} that on a set of full measure in $S^{n-1}$, $r_i(u) = \frac{1}{\kappa_i(x(u))}$ where $x(u)$ is the unique point in $\partial K$ with normal vector $u\in S^{n-1}$. 

Recall that for a normal point $x(u)$ the measure $\mu_K(x)$, pulled back to $S^{n-1}$, has density $\omega_n \prod_{i=1}^{n-1} r_i (u) \diff \sigma(u)$ for the normalized Haar measure $\sigma$ on $S^{n-1}$ and for $\omega_n = \vol_{n-1}(S^{n-1})$. In particular, letting $f_K(u)=\omega_n \prod _{i=1}^{n-1}   \kappa_i (K, x(u))^{-1} = \omega_n \prod _{i=1}^{n-1}r_i(u)$ for $u \in S^{n-1}$, we have $\diff \mu_K(u)=  f_K(u) \diff \sigma(u)$. Since normal points have full measure with respect to $\mu_K$ on $\partial K$, and since directions $u \in S^{n-1}$ where the support function $h_K$ is twice differentiable, in the sense of Alexandrov, also have full measure, throughout this note integration over $S^{n-1}$ should be interpreted as integration over the set of full measure $D_K \subseteq S^{n-1}$ of directions $u$ such that $x(\pm u)$ is normal in $K$ and also $h_K$ is twice differentiable at $u$. (See also the historical discussion in \cite[Section 10.5]{schneider2013convex} pertaining to Leichtweiss's generalization of the affine surface area to non-smooth bodies.)

Evidently we may replace the expression in Definition  
\ref{def:asa-new} by the following expression, which is easier to work with 
\begin{equation}\label{eq:omega-c-good-e}
\Omega^c(K) =
\omega_n \int_{S^{n-1}} \prod_{i=1}^{n-1}\left( 1 - {r_i (u)}\right)^\frac{1}{n+1} r_i^{\frac{n}{n+1}}(u)  \diff\sigma(u)
.
\end{equation}

\begin{rem}
In \cite{schuett2025floating} the authors actually consider for any $R>0$ the class of intersections of radius-$R$ balls, and have an $R$-version for the floating body and the affine surface area, which converges as $R\to \infty$ to the usual floating body and usual affine surface area. We chose to specialize to the case $R=1$ (as we did in \cite{AF-perrint}) since all results we consider can easily be rescaled.  
\end{rem}

Since clearly $\prod _{i=1}^{n-1} \left( \kappa_i (K, x) -1 \right)^\frac{1}{n+1}\le \prod _{i=1}^{n-1} \left( \kappa_i (K, x) \right)^\frac{1}{n+1}$, we have $\Omega^c(K) \le \Omega(K)$, and in particular the following version of an isoperimetric-type inequality for $\Omega^c$ was given in \cite{schuett2025floating}:
\begin{equation}\label{eq:affisop}
\Omega^c(K) \le \Omega (K)  \le n   \vol_n(B^n_2)^\frac{2}{n+1} \vol_n(K)^\frac{n-1}{n+1}.  
\end{equation}
We note that when $K$ is a very small ball then $r_i$ is small as well and we expect the inequality $\Omega^c(K) \le \Omega (K)$ to be close to being an equality. Since for a ball the second inequality in 
\eqref{eq:affisop} is equality, one does not expect to be able to improve the constant in \eqref{eq:affisop}. However, in contrast with the class $\K^n$, in the class 
  $\S_n$ one can consider the question of maximizing 
    $\Omega^c(K)$ 
  among all sets in $\S_n$, without a volume normalization. Indeed, it is easy to show that in $\S_n$, the body with largest $c$-affine surface area, is a ball of radius $1-\frac{1}{n+1}$.

  \begin{theorem}\label{thm:extremal-omegac}
Let $n\ge 2$. For any $K\in \S_n$ which is not a point, we have 
\[ \Omega^c(K) 
       \le \Omega^c \parens*{\frac{n}{n+1}B_2^n} ,\]
with equality if and only if $K = \frac{n}{n+1} B_2^n$.
\end{theorem}

The proof of Theorem \ref{thm:extremal-omegac} is actually one line, once we consider the expression \eqref{eq:omega-c-good-e}. Indeed, point-wise it is true that $(1-r)^\frac{1}{n+1} r^\frac{n}{n+1}$ is maximized, over $(0,1)$, for $r = \frac{n}{n+1}$, and therefore also the integral of this expression is maximal for this specific radius. 
Below we shall present yet another argument for Theorem \ref{thm:extremal-omegac} which captures some more intricate properties of the quantities involved.

\begin{rem}
We note that by definition 
$
\Omega^c(rB_2^n) 
=\omega_n 
   \left( 1-r \right)^\frac{n-1}{n+1} r^{\frac{n(n-1)}{n+1}}  
$ so that for any $K\in \S_n$
\[ 
\Omega^c(K) \le \Omega^c \parens*{\frac{n}{n+1}B_2^n} 
= \frac{n^\frac{n(n-1)}{n+1}}{(n+1)^{n-1}} \cdot \omega_n .
\]
\end{rem}

Our main  interest in this note is to compare $\Omega^c(K)$ and $\Omega^c(K^c)$. It turns out that a Santal\'o-type inequality holds for this pair, and our main result is the following.

\begin{theorem}\label{thm:aff-c-santalo}
Let $n\ge 2$. For any $K\in \S_n$ which is not a point or a translate of $B_2^n$, we have 
 \begin{eqnarray*}
    \Omega^c(K) \Omega^c(K^c)  &\le&  \Omega^c \parens*{\frac12 B_2^n}^2,
\end{eqnarray*}   
with equality if and only if $K = \frac12 B_2^n$.
\end{theorem}

As we shall demonstrate below, Theorem \ref{thm:aff-c-santalo} follows from an inequality connecting the quantities
$\Omega^c(K), \Omega^c(K^c)$ and the usual surface areas $S(K),S(K^c)$. Here $S(K)$ denotes the Lebesgue surface area of the convex set $K$. In particular $S(K) = \int_{\partial K} \diff \mu_K = \omega_n \int_{S^{n-1}} \prod_{i=1}^{n-1}r_i(u) \diff\sigma(u)$. Before stating this inequality in Theorem \ref{lem:mainlemthm}, let us describe the main tool from the study of the $c$-duality and boundary structure on $\S_n$ which will be used to prove it.

\begin{theorem}\label{thm:chakerian-better-fromulated}
     Let $n \geq 2$, $K\in \S_n$ which is not a point or a translate of $B_2^n$ and assume $u \in S^{n-1}$ is such that $h_K$ is twice differentiable at $u$ and $h_{K^c}$ is twice differentiable at $-u$, in the sense of Alexandrov. Then, letting 
     $0\le r_1\le \cdots \le r_{n-1}\le 1$ and $0\le s_1\le \cdots \le s_{n-1}\le 1$ stand for the principal radii of curvature of $K$ in direction $u$ and of $K^c$ in direction $-u$, respectively, we have 
     \[ r_i + s_{n-i} = 1, \quad i=1, \ldots, n-1.\]
 \end{theorem}

This theorem was given as Theorem 4.27 in \cite{AF-perrint} in the case where $h_K$ is twice \textit{continuously} differentiable, and generalizes a result by Bonnesen and Fenchel pertaining to a self dual body (i.e. a body of constant width). See the discussion in Bonnesen and Fenchel \cite[Chapter 15, Section 63 in page 128]{MR344997} and also Chakerian \cite{Chakerian}. Since the proof of Theorem~\ref{thm:chakerian-better-fromulated} is very similar to the proof of \cite[Theorem 4.27]{AF-perrint}, we defer it to the Appendix.

Using Theorem \ref{thm:chakerian-better-fromulated} we can prove our main tool, which we next state and prove.

\begin{theorem}\label{lem:mainlemthm}
Let $n\ge 2$ and let $K\in \S_n$ which is not a point or a translate of $B_2^n$. Then 
\[ \Omega^c(K) \le  S(K)^{\frac{n-1}{n}} \Omega^c (K^c)^{\frac1n} ,\]
with equality if and only if $K$ is a ball.
\end{theorem}

\begin{proof}
For   $u \in S^{n-1}$ where the points $x(\pm u)$ are both normal points of $\partial K$ and $\partial K^c$ respectively, and where $h_K$, $h_{K^c}$ are twice differentiable in the sense of Alexandrov at $u,-u$ respectively, let $\varphi(u) = \omega_n \prod_{i=1}^{n-1}   {(1-r_i(u))^{\frac{1}{n+1}}}({r_i (u)})^\frac{n}{n+1}$ and $\eta(u) = \parens*{\omega_n \prod_{i=1}^{n-1}   {(r_i(u))^{\frac{1}{n+1}}}(1-{r_i (u)})^\frac{n}{n+1}}^\frac1n$. By definition
\begin{eqnarray*}
\Omega^c(K) &=&  
\int_{S^{n-1}}\varphi(u) \diff\sigma(u)
\end{eqnarray*}
Using Theorem \ref{thm:chakerian-better-fromulated} we see that, after a change of variables $u \mapsto -u$, we have 
\begin{eqnarray*}
\Omega^c(K^c) &=&  
\omega_n \int_{S^{n-1}} \prod_{i=1}^{n-1}   {(r_i(u))^{\frac{1}{n+1}}}(1-{r_i (u)})^\frac{n}{n+1} \diff \sigma(u) = \int_{S^{n-1}} \eta^{n}(u) \diff\sigma(u) .
\end{eqnarray*}
We use the inequality of H\"older,
\begin{align*}
\int_{S^{n-1}}\varphi(u)  \diff\sigma(u) &= 
\int_{S^{n-1}} \frac{\varphi(u)}{\eta(u)} \cdot \eta(u) \diff\sigma(u) \\
&\leq
\left(\int_{S^{n-1}} \parens*{\frac{\varphi(u)}{\eta(u)}}^{\frac{n}{n-1}} \diff \sigma(u) \right)^{\frac{n-1}{n}} \left(\int_{S^{n-1}} \eta^{n}(u)  \diff\sigma(u) \right)^{\frac1n}. 
\end{align*}
By our choice of $\eta$ and $\varphi$, we see that 
\begin{eqnarray*}
\int_{S^{n-1}} \parens*{\frac{\varphi(u)}{\eta(u)}}^{\frac{n}{n-1}} \diff\sigma(u)
 &= &\int_{S^{n-1}} \left( \frac{ \omega _n \prod_{i=1}^{n-1}   {(1-r_i(u))^{\frac{1}{n+1}}}({r_i (u)})^\frac{n}{n+1}}{\omega_n^{1/n} \prod_{i=1}^{n-1}   {(r_i(u))^{\frac{1}{n(n+1)}}}(1-{r_i (u)})^\frac{1}{n+1}}
 \right)^{\frac{n}{n-1}} \diff\sigma(u)\\
   & = & \omega_n \int_{S^{n-1}}   \prod_{i=1}^{n-1}     {r_i (u)}  
  \diff \sigma(u) = S(K). 
 \end{eqnarray*}
Plugging into H\"older's inequality we    get 
\[ \Omega^c(K) \le S(K)^{\frac{n-1}{n}} \Omega^c(K^c)^{\frac1n}, \]
as claimed.

If we have equality, then we have equality in H\"older's inequality, which implies that $\frac{\varphi}{\eta}$ and $\eta$ are linearly dependent (as functions on $D_K$), which after a calculation implies that $\prod_{i=1}^{n-1} r_i(u)$ is constant on $D_K$, say $\prod_{i=1}^{n-1} r_i(u) \equiv R^{n-1}$, for some $R > 0$. Since $D_K$ has full measure, the surface area measures of $K$ and $RB_2^n$ are equal, and  by Minkowski's uniqueness theorem (see for example \cite[Theorem 8.1.1]{schneider2013convex} and the following paragraph), $K$ is a translate of $RB_2^n$.
\end{proof}

By multiplying the inequalities in Theorem \ref{lem:mainlemthm} and rearranging the terms we get 
\begin{cor}\label{cor:omega-santalo-vs-S}
Let $n\ge 2$. Then for any $K\in \S_n$ which is not a point or a translate of $B_2^n$, we have 
\[ \Omega^c(K)\Omega^c(K^c) \le  S(K)S(K^c),\]    
with equality if and only if $K$ is a ball.
\end{cor}

From Corollary \ref{cor:omega-santalo-vs-S} we can deduce Theorem \ref{thm:aff-c-santalo} easily, using that $K-K^c = B_2^n$ and that $S((1-t)K_0 + tK_1) \geq S(K_0 )^{1-t}S(K_1)^t$ for $K_0,K_1\in \K^n$ and $t\in (0,1)$ (which follows from Brunn-Minkowski's inequality for mixed volumes, see \cite[Theorem 20.4.1]{burago2013geometric} for example). 

\begin{proof}[Proof of Theorem \ref{thm:aff-c-santalo}]
    By Corollary~\ref{cor:omega-santalo-vs-S} we see that
    \[
        \Omega^c(K)\Omega^c(K^c) \le S(K)S(K^c)\le S \parens*{\frac12 (K-K^c)}^2= S \parens*{\frac12 B_2^n}^2 .
    \]
    If equality holds throughout, then it also holds in the first inequality, and then by Corollary~\ref{cor:omega-santalo-vs-S}, $K$ is a ball. Computing the functional in Theorem~\ref{thm:aff-c-santalo} for a ball $rB_2^n$, we see that $\Omega^c(rB_2^n) \Omega^c((1-r)B_2^n) = \omega_n^2 \parens*{r(1-r)}^{n-1}$, which can only equal $S \parens*{\frac12 B_2^n}^2 = \parens*{\frac{\omega_n}{2^{n-1}}}^2$ if $r = \frac12$. 
\end{proof}

Iterating the use of Theorem \ref{lem:mainlemthm} also produces the following lemma which is the key ingredient in our second proof of 
Theorem \ref{thm:extremal-omegac}, but can also be seen (given the Proposition \ref{prop:lastprop} which follows) as a strengthening of it.  

\begin{lem}\label{lem:maintoo}
Let $n\ge 2$ and let $K\in \S_n$ which is not a point or a translate of $B_2^n$. Then  
\[ \Omega^c(K) 
       \le
       S(K )^{\frac{n}{n+1} } 
       S(K^c)^{\frac1{n+1}}, \]
       with equality if and only if $K$ is a ball. 
\end{lem}

\begin{proof}
We use Theorem \ref{lem:mainlemthm} for both $K$ and $K^c$, recalling that $K^{cc} = K$, to write
\begin{eqnarray*}
    \Omega^c(K) &\le&  S(K)^{\frac{n-1}{n}} \Omega^c (K^c)^{\frac1n}\\
    \Omega^c(K^c) &\le&   S(K^c)^{\frac{n-1}{n}} \Omega^c (K)^{\frac1n}.
\end{eqnarray*}
Plugging the second into the first we get  
\begin{eqnarray*}
    \Omega^c(K) 
       &\le& S(K )^{\frac{n-1}{n}} 
      \left(S(K^c)^{\frac{n-1}{n}}\Omega^c (K)^{\frac1n}\right)^{\frac1n}.
\end{eqnarray*} 
Taking power $n$ and rearranging we end up with 
\begin{eqnarray*}
       \Omega^c(K)
       &\le&  S(K )^{\frac{n}{n+1} } 
       S(K^c)^{\frac1{n+1}} 
\end{eqnarray*}
as claimed. Since balls are the only equality cases in Theorem \ref{lem:mainlemthm}, they also are also the unique equality cases here.
\end{proof}

Finally we provide the following simple consequence of Alexandrov's inequality. Recall that Alexandrov's inequality, connecting various mixed volume ratios  (see, e.g. \cite[Theorem 1.1.11]{AGMBook}), states in particular that $\parens*{\frac{S(K)}{S(B_2^n)}}^{\frac{1}{n-1}} \le M^*(K) = \int_{S^{n-1}} h_K(u) \diff \sigma(u)$.

\begin{prop}\label{prop:lastprop}
Let $n \geq 2$. For any $K \in \S_n$,
  \[ S(K )^{\frac{n}{n+1} } 
       S(K^c)^{\frac1{n+1}} \leq S \parens*{\frac{n}{n+1} B_2^n}^\frac{n}{n+1} S \parens*{\frac{1}{n+1} B_2^n}^\frac{1}{n+1} .
    \]     
    Additionally, if there is equality in the inequality above, then $ K  = \frac{n}{n+1}B_2^n$. 
\end{prop}

\begin{proof}
By Alexandrov's inequality we have  
  \begin{eqnarray*}
S(K )^{n} 
       S(K^c) &\le& S(B_2^n)^{n+1} \left( M^*(K)^{n} M^*(K^c)    \right)^{n-1}. 
\end{eqnarray*}
Since $K - K^c = B_2^n$ we see that $M^*(K^c) = 1- M^*(K)$. Thus 
    \begin{eqnarray*}
S(K )^{n} 
       S(K^c) &\le&    
        S(B_2^n)^{n+1} \left( M^*(K)^{n} (1-M^*(K))   \right)^{n-1}
\\
       & \le & S(B_2^n)^{n+1} \left( \max_{r\in (0,1)} r^n (1-r)   \right)^{n-1}\\
       &= & S(B_2^n)^{n+1}\left(\frac{n^n}{(n+1)^{n+1}}\right)^{n-1}  = S \parens*{\frac{n}{n+1}B_2^n}^n S \parens*{\frac1{n+1} B_2^n},  
  \end{eqnarray*} 
completing the proof. If equality holds throughout, we must have equality in Aleandrov's inequality for mixed volumes, which implies $K$ is a Euclidean ball. Additionally, it means $M^*(K)^n (1 - M^*(K)) = \max_{r \in (0,1)} r^n (1-r)$, which means $M^*(K) = \argmax_{r \in (0,1)} r^n (1-r) = \frac{n}{n+1}$, completing the proof of the equality case. 
\end{proof}

\begin{proof}[Second Proof of Theorem \ref{thm:extremal-omegac}]
    It is a direct consequence of Lemma \ref{lem:maintoo} with Proposition \ref{prop:lastprop}, including the equality case.
\end{proof}

\begin{rem}
Considering the one-line proof of Theorem \ref{thm:extremal-omegac} presented at the beginning, one might expect a similar proof for Theorem \ref{thm:aff-c-santalo}. One way to go about this is to consider the expression $F(K) = \int_{S^{n-1}} \prod_{i=1}^{n-1} r_i(u)^{\alpha_i} (1-r_i (u))^{\beta_i} d\mu(u)$ where $\mu$ is any measure. If  $\alpha_i, \beta_i >0$ are chosen so that 
\[ \Phi(r_1, \ldots , r_{n-1}) =\frac12\left(\prod_{i=1}^{n-1} r_i^{\alpha_i} (1-r_i)^{\beta_i} + \prod_{i=1}^{n-1} (1-r_i)^{\alpha_i} r_i^{\beta_i}\right)\]
is maximized at the point $\bold{\frac12} = (\frac12, \frac12,\ldots ,\frac12)$, we  get the additive inequality
\[  \frac{F(K)  + F(K^c)}{2} = \int_{S^{n-1}}\Phi ((r_i (u))_{i=1}^{n-1})d\mu(u)\le \int_{S^{n-1}}\Phi \parens*{{\bold{\frac12}}} d\mu(u)  = F \parens*{\frac12 B_2^n}.
\]
(Which can be combined with $\sqrt{F(K)\cdot F(K^c)} \le \frac{F(K)  + F(K^c)}{2} $ of course.)
However, in the case which we have, namely $\alpha_i = \frac{1}{n+1}$ and $\beta_i = \frac{n}{n+1}$, 
the function $\Phi$ is not maximized at ${\bold {\frac{1}{2}}}$. Indeed, even considering only balls, taking ${\bf r} = (r, r, \ldots, r)$ we get  
\[
f(r) = 2\Phi({\bold r}) = r^{\frac{n(n-1)}{n+1}}(1-r)^{\frac{n-1}{n+1}} + r^{\frac{n-1}{n+1}}(1-r)^{\frac{n(n-1)}{n+1}}=r^a(1-r)^b + r^b (1-r)^a, 
\]
which, differentiating twice and plugging $r = \frac12$, gives
\[ f'' \parens*{\frac12} = \left( \frac12\right)^{a+b-3} \left((a-b)^2 -(a+b)\right).\]
 In our case,  $a - b = \frac{(n-1)^2}{n+1}$ and $a+b = n-1$ we see that for $n \geq 4$ the second derivative at $\frac12$ is positive, and so it is not a maximum. In other words, there is a ball $rB_2^n$ of some radius $r$ close to $\frac12$ with
 \[ \frac12 (\Omega^c (rB_2^n)  + \Omega^c ((rB_2^n)^c) ) > \Omega^c \parens*{\frac{1}{2}B_2^n} . \]
\end{rem}

\section*{\normalsize{Appendix: Proof of Theorem~\ref{thm:chakerian-better-fromulated}}}\label{sec:Appendix}

Let $K\in \S_n$ and assume $u\in S^{n-1}$ is such that $h_K$ is twice differentiable at $u$ and $h_{K^c}$ is twice differentiable at $-u$ in the sense of Alexandrov. Namely, a second order Taylor approximation of $h_K$ holds at $u$:
\[
    h_K(u+\delta v) = h_K(u) + \delta \nabla h_K(u)\cdot v + \frac{\delta^2}{2}v^T\nabla^2h_K(u)v+o(\delta^2) ,
\]
as $\delta \to 0$, and a similar Taylor approximation holds for $h_{K^c}(-u - \delta v)$. We used $\nabla^2 h_K(u)$ to denote the Hessian, according to Alexandrov. If $h_K$ is twice differentiable in the usual sense $\nabla^2 h_K(u)$ coincides with the usual Hessian. 

As $K\in \S_n$ is strictly convex, we have that  $h_K\in C^1(\RR^n \setminus \{0\})$. By definition of the $c$-dual, we know that for every $x \in \RR^n \setminus \{0\}$, $h_K(x)+h_{K^c}(-x) = |x|$ and therefore, differentiating,
\begin{equation}\label{eq:nablahhc}
    \nabla h_K(x) - \nabla h_{K^c}(-x) = \frac{x^T}{|x|} .
\end{equation} 

Since $h_K$ is twice differentiable at $u$ in the sense of Alexandrov, we can relate $\nabla^2 h_K(u)$ and $\nabla^2 h_{K^c}(-u)$ as follows.

\noindent Using $h_K(x) + h_{K^c}(-x) = |x|$, we see that
\begin{eqnarray*}
    |u+\delta v| &=& h_K(u+\delta v) + h_{K^c}(-u-\delta v) \\&=& 
    h_K(u) + \delta \nabla h_K(u)\cdot v + \frac{\delta^2}{2}v^T\nabla^2h_K(u)v \\&& +
    h_{K^c}(-u) - \delta \nabla h_{K^c}(-u)\cdot v + \frac{\delta^2}{2}v^T\nabla^2h_{K^c}(-u)v +
    o(\delta^2)\\
    & = & |u|+\delta (\nabla h_K(u)-\nabla h_{K^c}(-u))\cdot v+\frac{\delta^2}{2}v^T \parens*{\nabla^2h_K(u)+\nabla^2h_{K^c}(-u)} v+o(\delta^2) .
\end{eqnarray*}
Rearranging,
\begin{eqnarray*}
  \frac{  |u+\delta v|  - |u|}{\delta}  &=& (\nabla h_K(u)-\nabla h_{K^c}(-u))\cdot v+\frac{\delta}{2}v^T \parens*{\nabla^2h_K(u)+\nabla^2h_{K^c}(-u)} v+o(\delta) .
\end{eqnarray*}
Using the first and second order behavior of $|u+\delta v|$, we see
\begin{eqnarray*}
\frac12 v^T \parens*{\nabla^2 h_K(u) + \nabla^2 h_{K^c}(-u)} v = \lim_{\delta\to 0}  \frac{  |u+\delta v|  - |u| - \delta\frac{u^Tv}{|u|}}{\delta^2}  &=& \frac12 \frac{v^T}{|u|}\left(I - \frac{u}{|u|}\otimes \frac{u}{|u|}\right)v .
\end{eqnarray*}
Seeing $v \in \RR^n$ is arbitrary,
\begin{equation}\label{eq:hesses}
    \nabla^2 h_K(u) + \nabla^2 h_{K^c}(-u) =   \frac{1}{|u|}\left(I - \frac{u}{|u|}\otimes \frac{u}{|u|}\right) .
\end{equation}

Note that the matrix $\nabla^2h_K(u)$ is symmetric, and since $h_K$ is $1$-homogeneous, there is one eigenvector of eigenvalue $0$, which is in direction $u$. The other eigenvectors are in $u^\perp$, and are shared by $\nabla^2h_K(u)$ and $\nabla^2h_{K^c}(-u)$. By \eqref{eq:hesses}, we see that for any such eigenvector, the corresponding eigenvalues of $\nabla^2 h_K(u)$ and $\nabla^2 h_{K^c}(-u)$ sum to 1. Denoting by $r_1 \leq \ldots \leq r_{n-1}$ the eigenvalues of $\nabla^2 h_K(u)$ in increasing order, and by $s_1 \leq \ldots \leq s_{n-1}$ the eigenvalues of $\nabla^2 h_{K^c}(-u)$ in increasing order, we see indeed that $r_i$ and $s_{n-i}$ correspond to the same eigenvector and so $r_i + s_{n-i} = 1$.

\bibliographystyle{plain}

\end{document}